\documentclass[10pt]{amsart}
\usepackage{amssymb,amsmath,txfonts,amsthm}
\usepackage{hyperref}
\usepackage{mathrsfs}

\newcommand\blfootnote[1]{%
  \begingroup
  \renewcommand\thefootnote{}\footnote{#1}%
  \addtocounter{footnote}{-1}%
  \endgroup
}
\newtheorem{theorem}{Theorem}
\newtheorem{prop}{Proposition}
\newtheorem{lemma}{Lemma}
\newtheorem*{remark}{Remark}
\newtheorem{claim}{Claim}

\newtheorem{cor}{Corollary}

\numberwithin{equation}{section}
\newtheorem{question}{Question}

\def\XXint#1#2#3{{\setbox0=\hbox{$#1{#2#3}{\int}$}
  \vcenter{\hbox{$#2#3$}}\kern-.5\wd0}}

\author{Gang Liu}
\address{School of Mathematical Sciences, Shanghai Key Laboratory of PMMP, East China Normal University}
\email{gliu@math.ecnu.edu.cn}

\title[Complete K\"ahler manifolds with nonnegative Ricci curvature II]{Complete K\"ahler manifolds with nonnegative Ricci curvature II}
\date{}
\begin{document}
\begin{abstract}
We study rigidity on certain K\"ahler manifolds with nonnegative Ricci curvature. Among others things, we show that a complete noncompact K\"ahler surface with nonnegative Ricci curvature, Euclidean volume growth and quadratic curvature decay is biholomorphic to resolution of an affine algebraic variety. 

\end{abstract}
\maketitle

\blfootnote{The author is partially supported by National Key Research and Development Program of China (No. 2022YFA1005502), NSFC No. 12071140, the New Cornerstone Science Foundation through the New Cornerstone Investigator Program
and the Xplorer Prize, Science and Technology Commission of Shanghai Municipality (No. 22DZ2229014)}

\section{Introduction}

In \cite{[L2]}, the author studied certain complete K\"ahler manifolds with nonnegative Ricci curvature. The topic is divided into two parts: 

\medskip

(I)  when the holomorphic bisectional curvature is nonnegative, where the author established sharp estimates among curvature integral, dimension of polynomial growth holomorphic functions and the volume.

(II) Characterization of Ricci flat K\"ahler metrics from metric tangent cone at infinity.

\medskip 

In this note, we shall focus on the case when the Ricci curvature is nonnegative while the metric is not necessarily Ricci flat. Our results are motivated by a question of Ni \cite{[Ni]} , as well as conjecture $2$ in \cite{[L2]} on uniqueness of scalar curvature integral at infinity. Note also the integral of scalar curvature is also related with conjectures of Yau \cite{[Yau]}, Gromov \cite{[Gromov]}, Naber \cite{[Naber]}.  It is worth to point out that the uniqueness of scalar curvature integral also has topological application, see Corollary \ref{cor3}.

\begin{theorem}\label{thm1}
Let $M^n (n\geq 2)$ be a complete noncompact K\"ahler manifold with nonnegative Ricci curvature, Euclidean volume growth and quadratic curvature decay. Assume one tangent cone at infinity is Q-Gorenstein, then all tangent cones are Q-Gorenstein. Moreover, for all $k>0$, $\lim\limits_{r\to\infty}r^{2k-2n}\int_{B(p, r)}Ric^k\wedge \omega^{n-k}$ exists (finite).
\end{theorem}

\begin{cor}\label{cor1}
Let $M$ be as above. If $rank(Ric)\leq n - 2$ everywhere, then $M$ is Ricci flat.
\end{cor}

\begin{cor}\label{cor2}
Let $M$ be as above. If $rank(Ric)\leq n - 1$ everywhere and one tangent cone is Gorenstein, then $M$ is Ricci flat. 
\end{cor}

Let $(V, o)$ be a tangent cone of $M$ at infinity, let $r$ be the distance function to $o$. By Cheeger-Colding \cite{[CC1]} and that the curvature has quadratic decay condition,  $V$ is a K\"ahler cone with $C^{1, \alpha}$ metric. The reeb vector field $Jr\frac{\partial}{\partial r}$ defines a isometric and holomorphic action. Then one can decompose the holomorphic functions on $V$ into homogeneous pieces. Let the holomorphic spectrum be the set of $d$ such that $df = r\frac{\partial}{\partial r}f$ for some holomorphic function $f$ on $V$, not identically zero.

\begin{cor}\label{cor3}
Let $M^2$ be a complete noncompact K\"ahler surface with nonnegative Ricci curvature, Euclidean volume growth and quadratic curvature decay. Then all tangent cones have the same holomorphic spectrum. In particular, $M$ is biholomorphic to resolution of an affine algebraic variety. Moreover,  $\lim\limits_{r\to\infty}r^{-2}\int_{B(p, r)}S$ exists. 
\end{cor}
\begin{remark}
The algebraicity was proved in \cite{[Mok1]}\cite{[L3]}\cite{[CH2]}, where the Ricci curvature is assumed to be strictly positive or zero. In our case, Ricci is only assumed to be nonnegative. The new thing here is to show that tangent cones have the same holomorphic spectra, this was known before when the metric is Ricci flat \cite{[DS]} or the bisectional curvature is nonnegative \cite{[L1]}. The rigidity of holomorphic spectra seems unexpected, since it is known from Riemannian geometry that such thing can be rather flexible \cite{[CN]}.

\end{remark}

\begin{question}
Is it true that Corollary \ref{cor3} holds for higher dimensions?
\end{question}

While the general strategy follows from \cite{[L2]},  where $\int Ric^n$ plays a key role in connecting different tangent cones, there are extra difficulties, e.g., $Q$-Gorenstein condition does not ensure that one can pull back the nowhere vanishing holomorphic section to the manifold (nearby tangent cones). In particular, one might not have $\int_M Ric^n =  \int Ric_V^n$ in general. The key point is to show that $\int_M Ric^n = \int Ric_V^n$ modulo $c(2\pi)^n$, where $c$ is a rational number. Since the space of tangent cones is path connected, $\int_V Ric_V^n$ is independent of $V$. Another observation is that $\int_{B_V(o,1)}  Ric_V^k\wedge \omega_V^{n-k}$ is a geometric series. This gives the uniqueness of scalar curvature integral on space of tangent cones.

\medskip
\medskip

\begin{center}
\bf  {\quad Acknowledgments}
\end{center}
The author thanks Professor Chenyang Xu for some helpful discussion.

\section{The proofs}
\begin{proof}

Let $V$ be a tangent cone at infinity which is $Q$-Gorenstein. Since we assume quadratic curvature decay, the cross section has $C^{1, \alpha}$ regularity. Also, by 
\cite{[Ke]}, the cross section is a $RCD^*(2, 3)$ space. In particular, the fundamental group $G$ is finite.
Let $\Sigma$ be the universal cover of the cross section and let $V'$ be the cone over $\Sigma$. Then by analytic realization theorem of Stein on page 671 of \cite{[several]}, $V'$ can be realized as a normal complex analytic space. 
By $Q$-Gorenstein condition, we may assume $\Omega$ is an $m$-canonical section of $V$ which is nowhere vanishing. As in \cite{[MSY]} and Lemma $6.1$ in \cite{[CS]}, by multiplying by a nowhere vanishing holomorphic function, we may assume $\Omega$ is homogeneous.
Say $\log |\Omega|^2$ has degree $d$ with respect to $r\frac{\partial}{\partial r}$. Recall Lemma $5$ in \cite{[L2]}:
\begin{prop}\label{prop1}
Let $M$ be a complete noncompact K\"ahler manifold with nonnegative Ricci curvature, Euclidean volume growth and quadratic curvature decay. Then there exists $C_1>0$ such that for all $1\leq k\leq n$, for all $r>0$, 
$$r^{2k-2n}\int_{B(p, r)}Ric^k\wedge \omega^{n-k}<C_1.$$ Furthermore, if we consider a tangent cone $(V, o)$ and a sequence of rescaled manifolds $(M_i, p_i) = (M, p, \frac{r}{r_i})\to (V, o)$ in the pointed Gromov-Hausdorff sense, then for $k<n$, $$\lim\limits_{i\to\infty}r_i^{2k-2n}\int_{B(p, r_i)}Ric^k\wedge \omega^{n-k}=\int_{B_V(o, 1)}Ric_V^k\wedge \omega_V^{n-k},$$ where $Ric_V$ is understood in the current sense.
\end{prop}

\begin{claim}\label{cl1}
$d/m<C(M)$.
\end{claim}
\begin{proof}
Assume $(M_i, p_i) = (M, p, \frac{r}{r_i}) \to (V, o)$ in the pointed Gromov-Hausdorff sense.
According to Proposition \ref{prop1} and proposition $2.15$ of \cite{[L3]}, 
\begin{equation}\begin{aligned}C(M)&\geq \lim\limits_{i\to\infty}\int_{B(p_i, 1)}S_i \\&= \frac{1}{m}\int_{B(o, 1)}dd^c\log |\Omega|^2\wedge\omega_V^{n-1} \\&= \frac{1}{m}\int_{B(o, 1)}\Delta\log |\Omega|^2\omega_V^{n}\\&=
\frac{1}{m}\int_{\partial B(o, 1)}\frac{\partial \log |\Omega|^2}{\partial r}ds \\&= \frac{1}{m}c(n)dvol(B(o, 1)).
\end{aligned}\end{equation} Hence the claim is proved.
\end{proof}

As the cross section of $V'$ is simply connected, we can take the $m$-th root of $\Omega$ to obtain a canonical form $\eta$ which is homogeneous and nowhere vanishing. Now consider a sequence of tangent cones $V_i$ converging to $V$ in both metric and complex analytic sense. By the metric regularity, the cross sections are all diffeomorphic to that of $V$. In particular, we can take the universal cover, say $(V'_i, o_i)\to (V', o)$.  One can pull back $\eta$ to the canonical bundle of $V'_i\backslash{o_i}$ by local diffeomorphism. According to Example 9.4 of \cite{[Sa]}, 
$V'_i\backslash{o_i}$ admits a complete K\"ahler metric. As on \cite{[L2]} page 15-16, by suitable cut-off of the pull back section,  we can apply H\"ormander $L^2$ estimate of $\overline\partial$ on $V'_i\backslash{0}$ to get a nowhere vanishing section $\eta_i$ on the canonical bundle.
One can take the product $\prod\limits_{g\in G} g\eta_i$ to pass to the tangent cone $V_i$. By multiplying by a nowhere vanishing holomorphic function, we can assume it is homogeneous of uniform bounded degree ensured by Claim \ref{cl1}. This proves the openness of $Q$-Gorenstein condition on space of tangent cones. We can also assume such thing has norm $1$ on the unit ball. By boundedness of degree, the norm on the ball of radius $2$ is uniformly bounded. By using the same estimate as on page 17 of \cite{[L2]},  we have interior gradient estimates. 
By Cheeger-Gromov convergence, we proved the closedness of $Q$-Gorenstein condition along the path of tangent cones.
From now on, we may assume that $\Omega$ is a $k$-pluricanonical section where $k$ is the order of the fundamental group of $V$. Consider $(M_i, p_i) = (M, p, \frac{r}{a_i})$, where $a_i\to\infty$. 
\begin{lemma}\label{lm1}
Assume $B(p_i, 2)\to B_V(o, 2)$ in the Gromov-Hausdorff sense. Then modulo $(2\pi)^n\mathbb{Z}$, we have $$k^n\lim\limits_{i\to\infty}\int_{B(p_i, 1)} Ric_i^n  = \int_{B(o, 1)}(dd^c\log |\Omega|^2)^n.$$ 
\end{lemma}
\begin{proof}As in \cite{[L3]}, for large $i$, by the noncollapsing condition and quadratic decay of curvature, we may contract finitely many subvarieties and find a strongly pseudoconvex domain containing $B(p_i, 1)$. 
Then we can solve $\overline\partial$ problem. 
It is not clear whether there are pluri-canonical form $k>1$ on $B(p_i, 1)$ which is nowhere vanishing (due to the negative sign of curvature on pluri-canonical bundle, one may not solve $\overline\partial$ directly).
However, we can solve $\overline\partial$ on the canonical bundle to produce a nontrivial holomorphic $n$-form $s^i_1$ which may vanish somewhere. By choosing $s^i_1$ suitably, we can assume $s^i_1$ converges to $t_1$ as a canonical section on $B(o, 1)$.

On $B(p_i, 1)\backslash B(p_i, 0.5)$, there are only finitely many components of divisor $D^i_1$ (uniform in $i$). Then by Poincare-Lelong, \begin{equation}\frac{Ric_i}{2\pi} = \frac{1}{2\pi}dd^c \log |s^i_1|^2 - D^i_1 - E^i_1, \end{equation} where $E^i_1$ is supported in the exceptional divisor and $D^i_1$ is the strict transform. Below $\sim$ means equality modulo $(2\pi)^n\mathbb{Z}$. Let $\chi_i$ be a smooth nonnegative function which is equal to $1$ on $B(p_i, 1)$ and supported on $B(p_i, 2)$.
\begin{equation}\begin{aligned}\label{eq1}\int \chi_i Ric_i^n &= \int\chi_i (dd^c \log |s^i_1|^2 - 2\pi D^i_1 - 2\pi E^i_1)\wedge Ric_i^{n-1} \\& \sim \int\chi_i (dd^c \log |s^i_1|^2 - 2\pi D^i_1)\wedge Ric_i^{n-1}. \end{aligned}\end{equation} Note we have used that $2\pi\int_{E^i_1} Ric_i^{n-1}\in (2\pi)^n\mathbb{Z} $.
Using integration by parts, we find \begin{equation}\label{eq2}\begin{aligned}\int \chi_i dd^c \log |s^i_1|^2\wedge Ric_i^{n-1} & = \int \log |s^i_1|^2 dd^c\chi_i  \wedge Ric_i^{n-1} \\& \to \int \log |t_1|^2 dd^c\chi  \wedge \frac{(dd^c\log |\Omega|^2)^{n-1}}{k^{n-1}} , \end{aligned}\end{equation}
where we have used the fact that $dd^c\chi_i$ is supported on $B(p_i, 2)\backslash B(p_i, 1)$, which avoids the metric singularity under the convergence.

Again by solving $\overline\partial$ equation, we can just find a holomorphic section $s^i_2$ which is not vanishing on generic point of $D^i_1$, away from the exceptional set.
Similarly we can expand $Ric_i^{n-1}$ as $(dd^c \log |s^i_2|^2 - 2\pi D^i_2 - 2\pi E^i_2)\wedge Ric_i^{n-2}$. The second term of (\ref{eq1}) can be handled as follows:
\begin{equation}\label{eq3}\begin{aligned}\int -2\pi\chi_i D^i_1\wedge Ric_i^{n-1}
&=-2\pi \int_{D^i_1}\chi_i (dd^c \log |s^i_2|^2 - 2\pi D^i_2 - 2\pi E^i_2)\wedge Ric_i^{n-2}\\& \sim -2\pi \int_{D^i_1}\chi_i (dd^c \log |s^i_2|^2 - 2\pi D^i_2)\wedge Ric_i^{n-2}\end{aligned}\end{equation} We can do such thing $n$ times. Note finally there is a term given by the local intersection of $D^i_1, .., D^i_n$ which is an integer. Such thing is gone modulo $(2\pi)^n\mathbb{Z}$.  Therefore, the final result of (\ref{eq1}) consists of the sum of integrals whose support are uniformly away from the metric singularity.

Now we check the computation on the limit space. It is clear that $kRic_i\to dd^c\log |\Omega|^2$ as currents. We can compute \begin{equation}\int_{B(o, 1)}\chi(dd^c\log |\Omega|^2)^n = \int_{B(o, 1)}\chi  (kdd^c \log |t_1|^2 - 2\pi D'_1)\wedge(dd^c\log |\Omega|^2)^{n-1}, \end{equation} where $D'_1$ is the corresponding divisor which is Cartier, by our assumption that $\Omega$ is nowhere vanishing. Note the first term $\int_{B(o, 1)}\chi  kdd^c \log |t_1|^2\wedge(dd^c\log |\Omega|^2)^{n-1}$ is just $k^n$ times the last term in (\ref{eq2}), by integration by parts.
We can handle the second term  $\int_{B(o, 1)}- 2\pi D'_1\wedge (dd^c\log |\Omega|^2)^{n-1}$ similarly as in (\ref{eq3}). By repeating such $n$ times, we are left with the intersection of Cartier divisors $D'_1, ..., D'_n$, which gives an integer, also the sum of integrals which are supported away from the singularity. These are just limits of the quantities in (\ref{eq2}). We concluded the proof of Lemma \ref{lm1}.

\end{proof}

By Proposition \ref{prop1}, $\int_M Ric^n$ is finite. Note that the degree for the homogeneous $\Omega$ changes continuously when the tangent cone moves. The reason is that such degree is unique, also one can simply pass to subsequence to get the limit. Moreover, the degree for the homogeneous $\Omega$ determines $\int_{B(o, 1)}(dd^c\log |\Omega|^2)^n $, by the proof of Lemma 4 in \cite{[L2]}. Thus $\int_{B(o, 1)}(dd^c\log |\Omega|^2)^n$ changes continously on space of tangent cones. By applying Lemma \ref{lm1} to each tangent cone, we find that $\int_{B(o, 1)}(dd^c\log |\Omega|^2)^n$ is independent of tangent cones.  Let $\omega_0 = dd^c r^2$ be the metric on a tangent cone $V$. 

\begin{lemma}\label{lm2}
For $k = 0, 1, ..., n$, on a tangent cone $V$, $\int_{B(o, 1)}(dd^c\log |\Omega|^2)^k\wedge \omega_0^{n-k}$ is a geometric series with respect to $k$.

\end{lemma}
\begin{proof}
As in Lemma 4 of \cite{[L2]}, assuming $\log |\Omega^2| - c\log r^2$ is constant along $r\frac{\partial}{\partial r}$, then \begin{equation}\begin{aligned}\int_{B(o, 1)}(dd^c\log |\Omega|^2)^k\wedge \omega_0^{n-k} &= \int_{B(o, 1)}(dd^c\log |\Omega|^2)^k\wedge (dd^c \log r^2)^{n-k} \\& = c^k \int_{B(o, 1)} (dd^c \log r^2)^n \\& = c^k \int_{B(o, 1)} (dd^c r^2)^n, \end{aligned}\end{equation}
where the first and third equality use integration by parts, the second uses Lemma 3 of \cite{[L2]}.
\end{proof}

As we see before, when $k = n$ or $0$, $\int_{B(o, 1)}(dd^c\log |\Omega|^2)^k\wedge \omega_0^{n-k}$ is independent of tangent cones, hence by Lemma \ref{lm2},
$\int_{B(o, 1)}(dd^c\log |\Omega|^2)^k\wedge \omega_0^{n-k}$ is independent of tangent cones for all $k$. Proposition \ref{prop1} and the argument above prove that  $\lim\limits_{r\to\infty}r^{2k-2n}\int_{B(p, r)}Ric^k\wedge \omega^{n-k}$ exists (finite).
The proof of Theorem \ref{thm1} is complete.
\end{proof}

\medskip
\medskip

Proof of Corollary \ref{cor1}:

In this case, $\int Ric^{n-1}\wedge \omega = 0$. Proposition \ref{prop1} asserts that for $k<n$, $$\lim\limits_{i\to\infty}r_i^{2k-2n}\int_{B(p, r_i)}Ric^k\wedge \omega^{n-k}=\int_{B_V(o, 1)}Ric_V^k\wedge \omega_o^{n-k}.$$  Hence, by Lemma \ref{lm2}, we find $\int_{B_V(o, 1)}Ric_V^k\wedge \omega_o^{n-k}$ vanishes for all $k>0$. The corollary follows by taking $k=1$ and applying Corollary $5$ in \cite{[L2]}.

\medskip
\medskip

Proof of Corollary  \ref{cor2}:

In such case, $\int Ric^n = 0$. However, we are unable to apply Proposition \ref{prop1}  to conclude $\int_{B(o, 1)} Ric_V^n = 0$, since the limit identity holds only for $k<n$. With the Gorenstein condition, we can simply find nowhere vanishing canonical form on the geodesic ball sufficiently close a tangent cone. By adapting the argument in Lemma \ref{lm1}, we can prove that the integral of $Ric^n$ converges to that of the tangent cone (without modulo integers). The rest of the proof is same as Corollary \ref{cor1}.

\medskip
\medskip

Proof of Corollary \ref{cor3}:

Let $(V, o_V)$ be a tangent cone of $M$ at infinity. As we see before, the cross section has finite fundamental group $G$. Let $W\to V$ be the uniformising map. Then Mumford criteria states that $W$ is isomorphic to $\mathbb{C}^2$. Hence $V$ is a quotient singularity (see Proposition A.1 of \cite{[Zsx]}).  In particular, it is Q-Gorenstein.  The reeb vector field natually exists on $W$, and one can talk about polynomial growth holomorphic functions on $W$. We may assume $(z_1, z_2)$ form a homogeneous holomorphic coordinate on $\mathbb{C}^2$, say $r\frac{\partial}{\partial r} z_i = d_iz_i$. By similar calculations as in the proof of Theorem $3$ in \cite{[L2]}, $$vol(B_W(o_W, 1)) = \frac{\omega_4}{d_1d_2} = |G|vol(B_V(o_V, 1)), $$ $$\int_{B(o_W, 1)} Ric_W\wedge \omega_W = C(d_1+d_2-2)Vol(B(o_W, 1)) =  |G|\int_{B(o_V, 1)} Ric_V\wedge \omega_V.$$  It is known from volume comparsion and Theorem \ref{thm1} that these two quantities are independent of tangent cones. Therefore, $d_1, d_2$ are independent of tangent cones. Polynomial growth holomorphic functions on $V$ can be identified with polynomials on $W$ that are invariant under $G$. Hence, all tangent cones have the same holomorphic spectrum. According to Theorem $1.2$ in \cite{[L3]}, $M$ is biholomorphic to resolution of an affine algebraic variety.

\end{document}